\documentclass[11pt]{article} 

\usepackage{graphicx} 
\usepackage{amsmath,amssymb,amsfonts}  
\usepackage{color}
\usepackage{subfig}
\usepackage{epsfig}
\usepackage{pstricks}
\usepackage{epstopdf}
\usepackage{setspace}
\usepackage{authblk}

\usepackage{geometry}
\newcommand{\eome}{{\rm W\kern-.6em W}}

\usepackage[overload]{empheq}

\usepackage[bb=boondox]{mathalfa}

\newtheorem{de}{Definition}
\newtheorem{theo}{Theorem}
\newtheorem{proposition}{Proposition}
\newtheorem{corollary}{Corollary}
\newtheorem{lemma}{Lemma}
\newtheorem{remark}{Remark}
\newcommand{\Chi}{\mathfrak{X}}

\newcommand{\supeq}{\geqslant}
\newcommand{\infeq}{\leqslant}

\newcommand{\R}{\mathbb{R}}
\renewcommand{\S}{\mathbb{S}}

\newenvironment{proof}[1][]{\textbf{Proof #1:~}}{\hfill$\square$\\}

\renewcommand{\Chi}{\mathcal{X}}
\newcommand{\Nu}{\mathcal{V}}

\renewcommand\epsilon{\varepsilon}

\newcommand{\dnabla}{\nabla\!\!\!\nabla}

\title{Input / Output Stability of a Damped String Equation coupled with Ordinary Differential System\protect\thanks{This work is supported by the ANR project SCIDiS contract number 15-CE23-0014.}}
\author{}

\begin{document}





\maketitle

Matthieu Barreau$^{1*}$, \quad Fr\'ed\'eric Gouaisbaut${}^1$, \quad Alexandre Seuret${}^1$, and Rifat Sipahi${}^2$.\\
${}^1$ {LAAS - CNRS}, Universit\'e de Toulouse, CNRS, UPS, France (\textit{barreau,fgouaisb,seuret@laas.fr})\\
${}^2$ Department of Mechanical and Industrial Engineering, 334  Snell  Engineering  Center,  360  Huntington  Avenue,  Northeastern University, Boston, MA 02115 {USA} (\textit{rifat@coe.neu.edu})

\abstract{The input/output stability of an interconnected system composed of an ordinary differential equation and a damped string equation is studied. Issued from the literature on time-delay systems, an exact stability result is firstly derived using pole locations. Then, based on the Small-Gain theorem and on the Quadratic Separation framework, some robust stability criteria are provided. The latter follows from a projection of the infinite dimensional system states onto an orthogonal basis of Legendre polynomials. Numerical examples comparing these results with the ones in the literature are presented along with demonstrations of the effectiveness of the developed robust stability criteria.}


\section{Introduction}

In a number of complex systems, including cyberphysical systems, one particular interconnection that often arises includes an interaction between a finite and an infinite dimensional system, see for exemple, a flexible crane system \cite{HE2016146}, a drilling pipe \cite{bresch2014output}, temperature control \cite{tang2011state}, vibration in structures \cite{Wu20142787,7476820}, MEMS\cite{flores2014dynamics} among many others\cite{bastin2016stability}. In the presence of such interactions however, studying stability is a challenge since one cannot separately treat each system, but must consider them both in an ensemble\cite{Arcak2016,niculescu2001delay}. This paper is aimed at addressing this stability problem, specifically of an interconnection between an ordinary differential equation (ODE) and a partial differential equation (PDE). We consider here only the string equation, which falls into the one-dimension hyperbolic systems. Most of the work on the stability of this class of PDEs has been conducted considering a Lyapunov functional. The main techniques can be classified in two categories. Of the first kind is the backstepping methodology for boundary controlled PDEs. This approach has been utilized by many researchers, see a summary in\cite{krstic2009delay,krstic2011}. The other method relies on semi-group theory and uses an energy argument to conclude on stability. This method has also been widely studied, see, for example\cite{HE2016146,tucsnak2009observation,morgul1994dynamic,Morgül2002731}.

Noticing that the PDEs related to transport or string can be treated as a time-delay system (TDS), another approach considered in the study of stability of such systems relies on their input/output stability\cite{gu2003stability}. The derivation of a transfer function and the study of the characteristic equation forms the core of this approach toward assessing its input/output stability. On the wave equation, further results can be found in\cite{4099496,RNC:RNC1611,Grabowski2001,7984225} considering an input/output map or also using the port-Hamiltonian approach\cite{van2014port,califano2017stability}. On the other hand, in the cited references, the focus is on the PDE system, while the study of a coupled ODE/PDE system requires additional developments.

To address the stability problem of the coupled ODE/PDE system at hand, here we first make use of the comparison between TDS and string equation to obtain a transfer function representing the interconnected dynamics. An exact stability criterion similar to the one developed in \cite{olgac2002exact} is next adopted for this transfer function based on the pole locations of the system characteristic equation on the imaginary axis of the complex plane. Even if this approach provides exact results for stability, robustness and synthesis problems must be separately dealt with. To this end, a complementary approach is to use robust stability analysis tools as for example in \cite{espitia2016event} which considers an event-trigger control. In the same direction, the Small-Gain theorem \cite{curtain1995introduction} and Quadratic Separation\cite{peaucelle2007quadratic, 6760001} provide, as we show, efficient frameworks to reveal the stability conditions for the interconnected ODE/PDE system in a parameter region of interest.

The outline of the paper is as follows. Section 2 states the problem and provides some explanations regarding its physical interpretation. Section 3 deals with the input/output stability of the coupled system whose transfer function is derived and properties are deduced. Based on this, an exact stability criterion is developed in Section 4. Section 5 is dedicated to two tools used in robust stability analysis. More particularly, the Small-Gain theorem provides a very conservative but simple stability test while Quadratic Separation is developed in two stages: a first and then an extended analysis using projections of the infinite dimensional states of the ODE/PDE system. Finally, Section 6 presents examples and comparisons with other techniques.

The main contributions of this paper are as follows. The formulation of the ODE/PDE interconnection in the form of a TDS enables one to obtain some fundamental properties of the ODE/PDE system stability, and furthermore an algorithm for an exact stability criterion can be constructed as a baseline. Moreover, here, we propose a simple robust stability test based on Small-Gain Theorem from which many properties of the coupled system are deduced. Finally, the reformulation of the interconnected system in the Quadratic Separation framework enables a refinement of the previous stability results thanks to the Bessel inequality and Legendre polynomials. This new formulation shows that a new robust stability criterion can be obtained, which needs less computational resources than the one developed in \cite{besselString}.

\textbf{Notations:} The following notations are used: $\mathbb{R}^+ = [0, +\infty)$, $\mathbb{C}^+ = \left\{ z \in \mathbb{C} \ / \ \mathfrak{R}\text{e}(z) \supeq 0, z \neq 0 \right\}, L^2 = L^2([0, 1]; \mathbb{R}), {H}^n = \{ z \in L^2; \forall m \infeq n, \frac{\partial^m z}{\partial x^m} \in L^2 \}$ and $i$ is the imaginary number. $L^2$ is equipped with the norm $\|z\|^2 =  \int_0^1 \lvert z(x) \rvert^2  dx = \left<z,z\right>$.\\
For any square matrices $A$ and $B$, let $\text{diag}(A,B) = \left[ \begin{smallmatrix}A & 0\\ 0 & B \end{smallmatrix} \right]$. A symmetric positive definite matrix $P \in \S^n_+ \subset \R^{n \times n}$ is written as $P \succ 0$. For a matrix $A \in \mathbb{R}^{m \times p}$, then $A(1:N,:)$ (or, $A(:,1:N)$) is the sub-matrix of $A$ with the first $N$ lines (or, columns respectively). $A^{\perp}$ is the nullspace of $A$ and $*$ denotes the transconjugate. $0_{m,p}$ and $1_{m,p}$ are respectively the null and full of ones matrices with $m$ lines and $p$ rows and $0_{m} = 0_{m,m}$, $1_{m} = 1_{m,m}$. To keep the presentation concise, the following notations are used for partial derivatives of a function $u$ on an appropriate space: $u_t(x, t) = \frac{\partial u}{\partial t}(x, t),  u_{tt}(x, t) = \frac{\partial^2 u}{\partial t^2}(x, t), u_x(x, t) = \frac{\partial u}{\partial x}(x, t)$ and $u_{xx}(x, t) = \frac{\partial^2 u}{\partial x^2}(x, t)$.

\section{Problem Statement}

\subsection{Problem formulation}

The input/output stability of the following interconnected system is studied:
\begin{subequations}
	\begin{align}
		\dot{X}(t) &= AX(t) + B\left( u(1,t) + r(t) \right), & t \supeq 0, \label{eq:controller} \\
		u_{tt} (x,t) &= c^2 u_{xx}(x,t), & \!\!\!\!\!\!\!\!\!\!\!\! x \in [0, 1], t \supeq 0, \label{eq:wave} \\
		u(0,t) &= K X(t), & t \supeq 0, \label{eq:boundary1} \\
		u_x(1,t) &= -c_0 u_t(1,t), & t \supeq 0, \label{eq:boundary2} \\
		u(x, 0) &=u_0(x), & x \in [0, 1], \label{eq:initial1} \\
		u_t(x, 0) &= v_0(x), & x \in [0, 1], \label{eq:initial2} \\
		X(0) & = X_0, & \label{eq:initial3}
	\end{align}
	\label{eq:problem}
\end{subequations}
\!\!with the initial conditions $X_0 \in \mathbb{R}^n$ and $(u_0, v_0) \in {H}^2 \times H^1$ such that equations \eqref{eq:boundary1} and \eqref{eq:boundary2} are respected. Here, $(X_0, u_0, v_0)$ are compatibles with the boundary conditions and $r$ is the input function.

System~\eqref{eq:problem} represents the interconnection of a linear time invariant system with a string equation. Here, the state $u$ denotes the amplitude of the wave which belongs to a functional space and consequently is of infinite dimension. We assume that the state $u(x,t)$ belongs to $\mathbb{R}$ but the analysis can be extended without difficulty to $\mathbb{R}^m$, $m > 1$. The two subsystems are connected through a Dirichlet boundary condition, that is, the output $u(1,\cdot)$ and the input (eq. \eqref{eq:boundary1}) depend on the state $u$ at different positions but not on the derivative of $u$.\\
For the system to be well-posed, another boundary condition is also needed. In the literature, the choice is to use a boundary damping with equation \eqref{eq:boundary2}. Indeed, it has been shown in \cite{Lagnese1983163} for example that this condition ensures the stability of the wave equation itself if $c_0 > 0$ and the case $c_0 = 0$ removes the damping.

The wave equation operator associated with the above described boundary conditions is known to be diagonalizable (for more information on semi-group definitions and properties, the reader can refer to \cite{luo2012stability, tucsnak2009observation} and the references therein). Once diagonalized, it can be expressed as the composition of two transport equations, one going forward and another backward. Boundary condition \eqref{eq:boundary2} implies a reflection of the forward wave with a coefficient $\alpha = \frac{1 - c c_0}{1 + c c_0}$ as explained in \cite{louw2012forced}. Enforcing $c_0 > 0$ implies $|\alpha| < 1$ and consequently the energy of the wave is decreasing. These details will be the key in developing our results later in the paper.

It is critical to note that the string equation considered here behaves like a communication channel with a string dynamics. Note that most of the systems with a string equation however consider a collocated control\cite{morgul1994dynamic,6651788} which then yields a very different problem in nature compared to the one in the proposed interconnection. The purpose of this paper is to study the input/output stability of this interconnected system, where the definition of input/output stability for infinite dimensional systems follows from Definition 9.1.1 of \cite{curtain1995introduction} as:
\begin{de} System \eqref{eq:problem} is said to be \textbf{input/output stable} if for all energy bounded input $r$, the energy of the output $Y = KX$ is also bounded. \end{de}

\subsection{Existence and Uniqueness of the Solution}

Before going further, the existence of a solution to system \eqref{eq:problem} with $r = 0$ is proven. To do so, a step by step procedure is proposed. To ease the reading, the following sets are defined:
\begin{equation}
	\begin{array}{l}
		\mathbb{I}_n = [n c^{-1}, (n+1) c^{-1}), \quad n \in \mathbb{N}, \\
		\mathbb{H} = \left\{ (X, u, v) \in \mathbb{R}^n \times H^2 \times H^1 \right\}, \\
		\mathbb{D} = \left\{ (X, u, v) \in \mathbb{H} \text{ s.t. } u(0) = KX, u_x(1) = -c_0 v(1) \right\},
	\end{array}
\end{equation}
and using the same technique as in \cite{louw2012forced,evans2010partial} for $(X_0, u_0, v_0) \in \mathbb{D}$, a solution for $t \in\mathbb{I}_0$ is given by:
\begin{equation} \label{eq:u1tau}
	u(1,t) = u_0(1-ct) + \frac{\alpha}{2} \left( u_0(1-tc) + u_0(1) \right) + \frac{1+\alpha}{2} \!\! \int_{1-ct}^1\!\! v_0(s)ds.
\end{equation}

On $\mathbb{I}_0$, we notice that $u(1, \cdot) \in H^2(\mathbb{I}_0)$ depends only on the initial conditions. Then, $X$ on $I_0$ is the solution of the differential equation \eqref{eq:controller}, with initial condition $X(0) = X_0$. It is well known that this solution reads:
\begin{equation}
	\forall t \in \mathbb{I}_0, \quad X(t) = e^{At} X_0 + \int_0^t e^{A(t-s)} B u(1,s) ds.
\end{equation}

Then, $X \in C^1(\mathbb{I}_0)$  and consequently, $u(0,\cdot)$ has the same regularity property. On the domain $\Gamma_0 = \mathbb{I}_0 \times [0, 1]$, with the boundary conditions defined previously and using the formulas of \cite{louw2012forced}, Equation \eqref{eq:wave} has a unique solution $(x,t) \in \Gamma_0 \mapsto u(x,t) \in H^2$ and for $t \in \mathbb{I}_0$, $u_t(\cdot, t) \in H^1$. With all these considerations, it is possible to find $(X(c^{-1}), u(\cdot, c^{-1}), u_t(\cdot, c^{-1})) \in \mathbb{D}$ and then repeat the same procedure for $\mathbb{I}_1$ leading to the existence of a solution. We then get $(X, u, u_t) \in L^2( 0, +\infty, \mathbb{H})$. Moreover, $u$, $u_t$, $u_x$, $u_{tt}$ and $u_{xx}$ are $L^2$ on each compact set of $\mathbb{R}^+$. As the system is linear, energy consideration leads to the uniqueness property. If $r$ is not null, then we can calculate the steady-state $(X_f, u_f)$ and then the previous analysis holds for $(X-X_f, u-u_f)$.

\begin{remark} The set of solutions $\mathbb{H}$ is smaller than the one in \cite{besselString}. The main difference comes from considering the strong solution here while the weak solution will belong to the same space as in \cite{besselString}.\end{remark}

\section{Input / Output Analysis}

One possible way to study the input/output stability of system \eqref{eq:problem} is to study its characteristic equation, which can be obtained by applying the Laplace transform. This standard approach can be found in \cite{doetsch2012introduction, morgul1998stabilization}.

\subsection{Laplace transform of the wave equation}

The notion of Laplace transform extends easily from its traditional definition for finite dimensional systems as noted in \cite{doetsch2012introduction}. For this, the variables need to be $L^2$ on each compact set of $\mathbb{R}^+$, which is the case in \eqref{eq:problem} as discussed earlier.
%
%
Since we are not concerned with the transient response of the dynamics, we assume the initial conditions to be zero, i.e., $u^0 = v^0 = 0$. In Laplace domain, equation \eqref{eq:wave} can then be transformed into:
\begin{equation}
	\forall x \in [0, 1], \quad \quad s^2 U_{xx}(x,s) = c^2 U(x,s) ,
\end{equation}
with $s \in \mathbb{C}$ is the Laplace variable and $U(x, \cdot)$ is the Laplace transform of $u(x, \cdot)$. The solution defined on $x \in [0, 1]$ is given by:
\begin{equation}
	U(x,s) = C_1(s) e^{\frac{s}{c} x} + C_2(s) e^{-\frac{s}{c} x},
\end{equation}
where $C_1$ and $C_2$ are space-independent transfer functions to be determined using the boundary conditions.
%
Once calculated using equations \eqref{eq:boundary1} and \eqref{eq:boundary2}, the transfer function for the string equation is obtained as:
\begin{equation} \label{eq:ux}
	\mathcal{W}(x, s) = \frac{U(x,s)}{U(0,s)} = \frac{e^{-\frac{s}{c}x} + \alpha e^{\frac{s}{c}(x-2)}}{1 + \alpha e^{-2\frac{s}{c}}},
\end{equation}
with $\alpha = \frac{1-c c_0}{1 + c c_0}$ and $x \in [0, 1]$.
The transfer function from $U(0,s)$ to $U(1,s)$ then reads:
\begin{equation} \label{eq:W}
	\mathcal{W}(1, s) = \mathcal{W}(s) = \frac{U(1,s)}{U(0,s)} = \frac{(1+\alpha)e^{-s /c}}{1 + \alpha e^{-2s/c}} = \frac{2e^{-s / c}}{1 + c c_0 + (1 - c c_0) e^{-2 s / c}}.
\end{equation}
Notice that there are infinitely many poles of $\mathcal{W}(x,s)$. These poles are independent of $x \in [0, 1]$ and their real part $\frac{c}{2} \log \left\lvert \alpha \right\rvert$ is strictly negative if $c c_0 \neq 1$ and $c_0 > 0$.

\begin{remark} When $c_0 = 0$, $\mathcal W$ has infinitively many poles on the imaginary axis. Then Corollary~ 9.1.4 from \cite{curtain1995introduction} applies and there does not exist any finite-dimensional controller which exponentially stabilizes the wave equation. That is why only the case $c_0 > 0$ is considered.
\end{remark}

\subsection{Transfer function for the coupled system}

\begin{figure}
	\centering
	\includegraphics[width=7cm]{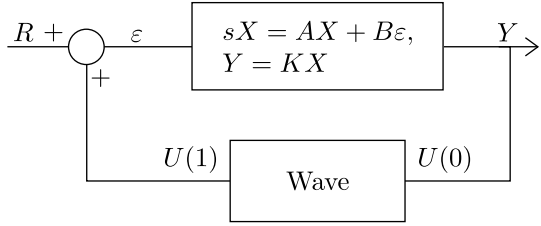}
	\caption{Block diagram for the system described by equation \eqref{eq:problem}.}
	\label{fig:block1}
\end{figure}

Considering the block diagram in Figure \ref{fig:block1}, the transfer function of the finite dimensional system is:
\begin{equation} \label{eq:H}
	\mathcal H(s) = \frac{Y(s)}{U(1,s) + R(s)} = K(sI - A)^{-1}B = \frac{N(s)}{D(s)} = \frac{l_{n-1} s^{n-1} + l_{n-2} s^{n-2} + \dots + l_0}{m_{n} s^n + m_{n-1} s^{n-1} + \dots + m_0},
\end{equation}
where $N$ and $D$ are real polynomials in $s$ with $m_n \neq 0$ and $R$ is the Laplace transform of the reference signal $r$ in \eqref{eq:problem}. \\
Using transfer function $\mathcal W$ developed previously, we then obtain the following closed-loop system transfer function:
\begin{equation} \label{eq:F}
	\mathcal F(s) = \frac{Y(s)}{R(s)} = \frac{\mathcal H(s)}{1- \mathcal H(s) \mathcal W(s)} = \frac{N(s) \left( 1+ \alpha e^{-2 s / c} \right)}{(1+c c_0) c_{eq}(s, c)},
\end{equation}
with
\begin{multline} \label{eq:CE}
	c_{eq}(s, c) = \left( 1+ c c_0 + (1 - cc_0) e^{-2s / c} \right) D(s) - 2 N(s) e^{-s / c} \\
	= m_n \left( 1+ c c_0 + (1 - cc_0) e^{-2s / c} \right) e^{-2s / c} s^n + \sum_{k=0}^{n-1}  \left( \left( 1+ c c_0 + (1 - cc_0) e^{-2s / c} \right) m_{k} - 2 l_k \right) s^k  .
\end{multline}

The stability of the dynamics is determined by the location of the zeros of \eqref{eq:CE} on the complex plane as we discuss in detail in the following subsection.

\subsection{Input / Output stability}
A time-delay system is of \emph{neutral type} when its highest order derivative is affected by a delay term $\tau$. Given a characteristic equation, this feature appears as the highest power of the Laplace variable $s$ multiplying the time delay operator $e^{-\tau s}$. Similarly, in state-space representation, neutral property is easily detected by the presence of the derivative of the state term being affected by delays. For a complete treatment of such systems, see details on neutral functional differential equations (NFDE) in\cite{niculescu2001delay,kolmanovskii2013introduction}.

Inspecting the corresponding transfer function $\mathcal F$ or the characteristic equation \eqref{eq:CE} of system \eqref{eq:problem}, it is easy to see that this system exhibits the neutral type property. That is, from an input/output approach, system \eqref{eq:problem} is classified as a neutral system. Moreover, we can re-write system \eqref{eq:problem} in the standard NFDE form by collecting the derivative of the states on the left hand side:
\begin{equation} \label{eq:ss}
	\dot{X}(t) + \alpha \dot{X}(t - 2 c^{-1}) = A X(t) + \frac{2}{1 + c c_0} BK X(t - c^{-1})  + \alpha AX(t - 2 c^{-1})  + Br(t) +  \alpha B r(t-c^{-1}),
\end{equation}
where from \cite{hale1977theory,olgac2004practical} the difference operator is defined as 
\[
	D(X)(t) = X(t) + \alpha X(t-2c^{-1}),
\]
for all $t \in \R^+$.

\begin{remark} One can see that for $c c_0 = 1$, we have $\alpha = 0$ then system \eqref{eq:problem} is no longer a neutral type system. It is of \emph{retarded type} \cite{bellman1963differential}. This distinction is critical when studying the stability of TDS. It will also help build various example cases and enable a comparison with the stability tests already available for TDS. \end{remark}

When it comes to studying the stability of NFDEs, contrary to functional differential equation of retarded type, inspecting the pole locations of the corresponding characteristic equation is not sufficient to guarantee input/output stability. This is due to the well-known small-delay phenomenon, also known as small $\tau$-stabilizability. When an NFDE is small $\tau$-stabilizable, then this guarantees that the poles of the NFDE behave in a continuum as the delay value changes from $0$ to $0^+$. In other words, the stability properties of the NFDE are preserved when the delay is infinitesimally increased past zero. Whenever this property holds, then studying the pole locations of an NFDE allows one to conclude on the stability of an NFDE (see \cite{bellman1963differential,opac-b1084291} for more information).

\begin{de} An NFDE system is said to be \textbf{small $\tau$-stabilizable} if the difference operator $D(X)(t)$ is stable. \end{de}

This definition comes from \cite{hale1977theory,olgac2004practical}. According to Theorem 12.5.1 and Corollary~12.5.1 from \cite{hale1977theory}, a necessary and sufficient condition for the difference operator to be stable is that $|\alpha|$ is strictly less than $1$.

\begin{proposition} \label{lem:tauStabilizable} System \eqref{eq:problem} is neutral and small $\tau$-stabilizable if and only if $c_0 > 0$ and $c c_0 \neq 1$. \end{proposition}

Finally, since the system is small $\tau$-stabilizable under a given condition, the input/output stability can be discussed.
\begin{corollary} If $(X_0, u_0, v_0) \in \mathbb{D}$ and all the poles of $\mathcal F$ as defined in equation \eqref{eq:F} are on the left half of the complex plane, then system \eqref{eq:problem} is input/output stable. The equilibrium points $(X_e, u_e, v_e) \in \mathbb{H}$ are such that $(A+BK)X_e = 0$, $u_e = K X_e$ and $v_e = 0$. \\
Moreover, if the system $\Sigma(A, B, K)$ is stabilizable and detectable such that $A+BK$ is not singular, then system \eqref{eq:problem} is asymptotically stable.
\end{corollary}
\begin{proof}
	If system \eqref{eq:problem} is small $\tau$-stabilizable, then it is input/output stable if all the poles of $\mathcal F$ are with a strictly negative real part (Theorem 9.9.1 from \cite{opac-b1084291}). The link between asymptotic and input/output stability comes from the pole/zero cancellation. If the system $\Sigma(A, B, K)$ is stabilizable and detectable, then, the pole/zero simplification in $\mathcal F$ are with negative real parts. It means that the non-observable or non-controllable parts of system \eqref{eq:problem} are stable. Then, all states will be indeed converging. The study of equilibrium points has been established in Proposition 2 of \cite{besselString} and if $A+BK$ is not singular, state $X$ converges toward zero and the system is asymptotically stable.
\end{proof}

Now that we established some key properties regarding the characteristics of \eqref{eq:problem}, the stability analysis can be pursued. We consider two main approaches, one using pole locations and another with robust stability tools.

\section{A frequency domain argument on stability}

Once the characteristic equation is established, different methodologies can be applied to assert the input/output stability of system \eqref{eq:problem}. Various techniques can be adopted for this purpose, see for example \cite{sipahi2011stability}, including those where $c_{eq}(\cdot, c)$ in \eqref{eq:CE} has the coefficients explicitly depending on the delay $c^{-1}$.This problem has been investigated in some studies\cite{beretta2002geometric,jin2018stability}. 
In this part, we focus on the exact treatment of this problem coming from the time-delay system literature. Specifically, a pole location argument is taken based on frequency domain tools in order to assess stability. We study the stability problem using the Cluster Treatment of Characteristic Roots (CTCR) methodology originally proposed in \cite{olgac2002exact}. Let $\tau = c^{-1}$ and regrouping the terms by their delay dependence, characteristic equation \eqref{eq:CE} becomes:
\begin{equation} \label{eq:a}
	c_{eq}(s, \tau) = a_0(s, \tau) + a_1(s, \tau) e^{-\tau s} + a_2(s, \tau) e^{- 2 \tau s},
\end{equation}
with $a_0(s, \tau) = \left(1+ \frac{c_0}{\tau} \right) D(s)$, $a_1(s, \tau) = -2 N(s)$ and $a_2(s, \tau) = \left(1- \frac{c_0}{\tau} \right) D(s)$. We then claim that the system described by transfer function $\mathcal{F}$ can switch from stable to unstable behavior, or vice-versa, for a given delay only if \eqref{eq:CE} has a pole on the imaginary axis\cite{olgac2002exact}.

CTCR starts by exhaustively detecting all the imaginary axis poles $s = i \omega$, along with their corresponding delay values (Proposition 1 of \cite{olgac2002exact}). Next, CTCR identifies that each pole $s = i \omega$ has a unique crossing direction over the imaginary axis for all the delays creating this crossing (Proposition 2 of \cite{olgac2002exact}). Knowing the number of poles at $\tau = 0$, which is trivial to assess, it is then possible to use the information regarding crossing directions and at which delay values such crossings occur to track the pole locations across the imaginary axis. With this idea, it becomes possible to count the number of poles on the right-half plane for a given delay value $\tau > 0$. Whenever there are no unstable poles for certain delays, we then state that the system at hand is input/output stable for those delays.

CTCR framework has already been demonstrated on neutral systems\cite{olgac2004practical} while respecting the small $\tau$-stabilizability property. In the following, we briefly summarize this framework from the cited studies and also point out the main differences introduced while studying the particular system \eqref{eq:problem}. Following the process described in \cite{olgac2002exact,olgac2004practical}, the Rekasius substitution\cite{rekasius} is defined as:
\begin{equation} \label{eq:rekasius}
	e^{- \tau s} := \frac{1 - Ts}{1 + Ts}, \quad \tau \in \mathbb{R}^+, \quad T \in \mathbb{R}, \quad s = i \omega,
\end{equation}
which is an \textit{exact} substitution of exponential terms when $s = i \omega$. This substitution is different than Pad{\'e} approximation since in general $T \neq \tau / 2$.
Next, substituting \eqref{eq:rekasius} into the characteristic equation \eqref{eq:CE} and expanding by $(1+Ts)^2$, which does not bring any artificial imaginary poles, we obtain a \emph{transformed} characteristic equation:
\[
	\bar{c}_{eq}(s,T) = \left( 1+ \cfrac{2 c_0}{\tau} T s + T^2 s^2 \right) D(s) - N(s) (1 - T^2 s^2),
\]
which is nothing but a multinomial. Moreover the imaginary poles $s = i\omega$ of the original characteristic equation and this multinomial are identical\cite{olgac2002exact} and hence one can alternatively compute the imaginary poles $s = i\omega$ from this multinomial, which is a much easier task. To this end, first build the coefficients $\left\{ b_k(T, \tau) \right\}_{k \in (0, n+2) }$, such that:
\begin{equation} \label{eq:b}
	\bar{c}_{eq}(s, T) = \sum_{k=0}^{n+2} b_k(T, \tau) s^k.
\end{equation}

Next notice that the presence of delay $\tau$ as a coefficient in \eqref{eq:b} is not a standard form, which may complicate the computation of imaginary poles and the application of Proposition 2 in \cite{olgac2002exact} as coefficients influenced by delays may also rule out certain periodicity properties. This issue is removable in the particular problem at hand. We consider the following manipulation: as $c_0$ is always divided by the delay $\tau$ in $b_k$, if one defines a new positive variable $c_1 = c_0 \tau^{-1}$, then $b_k$ depends only on $c_1$ and not on $\tau$ anymore. Working at a given strictly positive $c_1$ removes the dependence of $b_k$ in $\tau$. Since they are now independent, the methodology still applies. This manipulation also shows that $c_1$ is a variable of interest when it comes to studying the parametric stability boundaries of damped waves.

Using $b_k$ and $a_k$ as defined in equations \eqref{eq:a} and \eqref{eq:b}, the CTCR methodology provides the boundaries of the parameter space in which the input-output stability of system \eqref{eq:problem} holds\footnote{Here, we provide the general framework of CTCR for an exact stability analysis. Handling degenerate cases requires care as was demonstrated in \cite{sipahi2003degenerate}. Such special cases, for a slightly different $c_1$, will however not arise.}. This method is summarized as follows for a given $c_1 =c_0 \tau^{-1}$:
\begin{enumerate}
	\item Using $\bar{c}_{eq}$, find the roots $s = i \omega$ corresponding to $T \in \mathbb{R}$, e.g., by using Routh's array. There are only a finite number of such solutions (Proposition 1 of \cite{olgac2002exact}) ;
	\item For each $T \in \mathbb{R}$ obtained previously, we already have $\omega$ where a crossing on the imaginary axis exists. For each $\omega$, there is a root tendency, indicating the unique direction of the crossing independent of the delays creating that crossing (Proposition 2 of \cite{olgac2002exact}).
	\item Then, using the inverse transformation of Rekasius transformation in \eqref{eq:rekasius}, it is possible to find all the delays $(\tau_{\ell})$ corresponding to each pair of $(T, \omega)$ and falling in an interval from $0$ up to a target delay value $\tau_{max}$. Sorting these delays in ascending order, and starting with the number of unstable roots for $\tau = 0$, the number of unstable roots for a delay $\tau < \tau_{max}$ can be accounted by observing the root tendency property of the crossings.
	\item The stability areas of system \eqref{eq:problem} for a given $c_1 = c_0 \tau^{-1}$ are when no unstable poles are detected. Since $\tau$ is known at this point, $c_0$ can be recovered: $c_0 = c_1 \tau$.
\end{enumerate}

We point out that the coefficients $b_k$ depend linearly on $c_1$ and then the roots of $\bar{c}_{eq}(\cdot, T)$ vary continuously subject to $c_1$. In other words, the delays $\tau$ for which there is a crossing vary continuously relative to $c_1$. The border of each stability area on the map $(c_1, \tau)$ is consequently continuous. This is one of the arguments for considering the mapping $(c_1, \tau)$ instead of $(c_0,c)$.


\section{Robust Stability Analysis}

The previous stability analysis results are exact, however, we still need to develop efficient tools to deal with robustness issues with respect to, for example, an uncertainty on $A$ or $c$. Here, we aim at obtaining tools not only for the specific problem \eqref{eq:problem}, but for a general interconnection dynamic. That is why tools coming from the robust analysis are also considered. In \cite{besselString}, the authors built a Lyapunov functional to ensure the exponential stability and the main drawback was an important number of decision variables making the treatment computationally demanding for large scale systems. This section aims at providing analysis tools with similar results but of lower computational complexity.

Since the studied system is an interconnection between two subsystems, here, the stability of the interconnection is stated under some conditions on each subsystem. Two tools coming from the robust analysis are considered: Small Gain theorem and Quadratic Separation.

\subsection{Small-Gain Theorem}
We consider the block diagram of Figure \ref{fig:block1} where the wave equation is treated as a disturbance. The transfer functions of the disturbance and the plant are borrowed from Section 3. The following stability criteria is a direct application of the Small-Gain Theorem.

\begin{theo} \label{smallGain}
	Let system \eqref{eq:problem} with $A$ Hurwitz with its $H_{\infty}$ less than $1$. The system is input-output stable for $(c, c_0)$ if the following condition is satisfied:
	\begin{equation} \label{eq:smallGain}
		\| \mathcal H \|_{\infty} = \max_{w \in \mathbb{R}^+} |\mathcal{H}(iw)| < c c_0.
	\end{equation}
\end{theo}

\begin{remark} Note that the condition $c_0 > 0$ (and consequently $|\alpha| < 1$) is equivalent to $\| \mathcal W \|_{\infty}$ bounded.\end{remark}

\begin{proof}
Beforehand, the infinity norm of the disturbance is computed. Some calculations lead to $\| \mathcal W \|_ {\infty} = \frac{2}{(1+c c_0) \min_{\omega \supeq 0} \left\lvert 1 + \alpha e^{-2 i \omega / c} \right\rvert}$.
For $|\alpha| < 1$, the function $\omega \mapsto 1 + \alpha e^{-2 i \tau \omega}$ is inside the circle on the complex plane centered at $1$ and of radius $\left\lvert  \alpha \right\rvert  < 1$. The minimum of $\| \mathcal W \|_{\infty}$ is $\max \left( (c c_0)^{-1} , 1 \right)$.\\
Then, using the small gain theorem for infinite dimensional systems as stated in Theorem~9.1.7 by  \cite{curtain1995introduction}, the interconnected system is asymptotically stable if $\| \mathcal H \|_{\infty} \| \mathcal W \|_{\infty} < 1$ and each subsystem is stable. The second condition is ensured if $\| \mathcal H \|_{\infty} < 1$ and $c_0 > 0$. Considering the case $c c_0 \infeq 1$ leads to $\| \mathcal W \|_{\infty} = (c c_0)^{-1}$, it then follows that robust stability is ensured. If $c c_0  > 1$, then the stability is ensured if $\| \mathcal H \|_{\infty} < 1$, which is true by assumption.
\end{proof}

This theorem is quite conservative, mostly because the Small-Gain Theorem provides only a sufficient condition. However the resulting stability test is simple and some properties can be deduced.

\begin{proposition}
	If $\| \mathcal H \|_{\infty} < 1$, then there exists a function $c_0 \mapsto c_{min}(c_0)$ where $c_{min}(c_0) \infeq \| \mathcal H \|_{\infty} c_0^{-1}$. That leads to three properties:
		\begin{enumerate}
			\item Since $\| \mathcal H \|_{\infty} < 1$, $A+BK$ is stable,
			\item  For a given $c_0 > 0$,  system \eqref{eq:problem} is stable for all $c \supeq c _{min}(c_0)$,
			\item $\displaystyle \lim_{c_0 \to + \infty} c_{min}(c_0) = 0$.
		\end{enumerate}
\end{proposition}

To reduce the conservatism introduced in this subpart, another framework is proposed.

\subsection{Quadratic Separation - Preliminary result}

The Small-Gain Theorem ensures the exponential stability of an interconnected system composed of a disturbance and a plant both stable. To decrease the conservatism and consider a broader class of interconnected systems, the Quadratic Separation (QS) framework can be used. This framework has been originally proposed in \cite{668829} and it studies the well-posedness of a closed-loop system made up of an unknown disturbance and a plant.\\
We describe the methodology of QS to provide a preliminary result on system \eqref{eq:problem} and then we extend this stability analysis to a more general case. QS states the well-posedness of a generic system described in Figure \ref{fig:sepQuad}, where $\nabla$ is called the uncertainty matrix and belongs to a set $\dnabla$. The well-posedness is defined as follows:
\begin{de}
	The interconnected system described on Figure \ref{fig:sepQuad} is well-posed with respect to the norm $\| \cdot \|$ if
	\begin{equation}
		\exists \gamma > 0, \forall~\nabla \in \dnabla, \forall~\bar{\omega}, \bar{z}, \quad \left\| \left[ \begin{matrix} \omega \\ z \end{matrix} \right] \right\| \leq \gamma \left\| \left[ \begin{matrix} \bar{\omega} \\ \bar{z} \end{matrix} \right] \right\|,
	\end{equation}
	where $\bar{\omega}$ and $\bar{z}$ are the references.
\end{de}
\begin{figure}
	\centering
	\includegraphics[width=5cm]{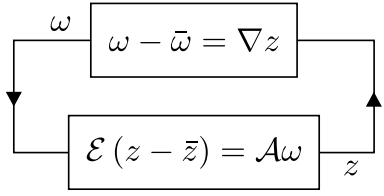}
	\caption{Feedback system for quadratic separation. $\mathcal{A}$ and $\mathcal{E}$ are matrices and $\nabla \in \dnabla$ is the uncertainty. $\bar{\omega}$ and $\bar{z}$ are the references.}
	\label{fig:sepQuad}
\end{figure}
In our case here, we consider the disturbance to be an operator acting on two signals $w$ and $z$ in Laplace domain. In other words, the following system is obtained:
\begin{equation} \label{eq:qsframework}
	\left\{
	\begin{array}{l}
		\Omega(s) = \nabla(s) Z(s), \\
		\mathcal{E} Z(s) = \mathcal{A} \Omega(s),
	\end{array}
	\right.
\end{equation}
with $\mathcal{E}$ full column rank, $Z$ and $\Omega$ are the Laplace transformations of $z$ and $\omega$, respectively and $\mathcal{A}$ is a description of the system, its explicit form is given after. The well-posedness of system \eqref{eq:qsframework} is assessed in Theorem 1 and Corollary 2 from \cite{peaucelle2007quadratic}. Following this formulation, the next theorem is stated.
\begin{theo} \label{sec:theoQS_general}
	The system described by Figure \ref{fig:sepQuad} and Equation \eqref{eq:qsframework} is well-posed if and only if there exists a real matrix of appropriate dimension $\Theta = \Theta^{\top}$ such that:
\begin{equation} \label{eq:separateur}
	 \forall s \in \mathbb{C}^+, \quad \left[ \begin{matrix} I \\ \nabla(s) \end{matrix} \right]^* \Theta \left[ \begin{matrix} I \\ \nabla(s) \end{matrix} \right] \preceq 0,
\end{equation}
\begin{equation} \label{eq:LMI}
	{\left[ \begin{matrix} \mathcal{E} & - \mathcal{A} \end{matrix} \right]^{\perp}}^{\top} \Theta \left[ \begin{matrix} \mathcal{E} & - \mathcal{A} \end{matrix} \right]^{\perp} \succeq 0,
\end{equation}
where $\nabla^*$ is the trans-conjugate of $\nabla$ and $\mathcal{E}$ is full-rank.
\end{theo}

The previous theorem has been adapted to our system considering $\nabla$ is an operator depending only on the Laplace variable $s$ such that $\dnabla = \left\{ \nabla(s) | s \in \mathbb{C}^+ \right\}$. In this case, the well-posedness of system \eqref{eq:qsframework} implies its input/output stability. Indeed, the well-posedness of system \eqref{eq:qsframework} with $\nabla \in \dnabla$ ensures that there is a unique solution and $\omega$ and $z$ are bounded by the exogenous signals $\bar{\omega}$ and $\bar{z}$. Then, there are no poles with a strictly positive real part and consequently, the system is input/output stable.

Now, we can transform the block diagram in Figure \ref{fig:block1} into a suitable form to apply Theorem \ref{sec:theoQS_general}. Consider the following signals:
\begin{equation}
	z(t) = \left[ \begin{matrix} \dot{X}^{\top}(t) \ & K X(t) & \ K X (t-\tau) & \ K \dot{X}(t) \end{matrix} \right]^{\top}\!, \quad \quad \omega(t) = \left[ \begin{matrix} X^{\top}(t) \ & K X(t-\tau) \ & u(1,t) \ & K ( X(t) - X(t-\tau) ) \end{matrix} \right]^{\top}.
\end{equation}

For the equivalence between \eqref{eq:qsframework} and transfer function \eqref{eq:F}, $\nabla, \mathcal{A}$ and $\mathcal{E}$ are defined as follows:
\begin{equation} \label{eq:def1QS}
	\begin{array}{ll}
		\begin{array}{l}
			\forall s \in \mathbb{C}^+, \nabla(s) = \text{diag}\left( s^{-1}I_n, e^{-\tau s}, \delta(s), \delta_0(s) = \frac{1-e^{-\tau s} }{s} \right), \\
			\delta(s) = \cfrac{1+ \alpha}{1 + \alpha e^{-2\tau s}},
		\end{array} & \quad
		\mathcal{E} =\left[ \begin{matrix} I_n & 0_{n, 1} & 0_{n, 1} & 0_{n, 1} \\ 0_{1,n} & 1 & 0 & 0 \\ 0_{1,n} & 0 & 1 & 0 \\ -K & 0 & 0 & 1 \\ 0_{1,n} & 1 & -1 & 0 \end{matrix} \right], \quad \mathcal{A} = \left[ \begin{matrix} A & 0_{n, 1} & B & 0_{n, 1}  \\ K & 0 & 0 & 0 \\  0_{1,n} & 1 & 0 & 0 \\ 0_{1,n} & 0 & 0 & 0 \\ 0_{1,n} & 0 & 0 & 1 \end{matrix} \right].
	\end{array}
\end{equation}

Note that $\mathcal{E}$ is full column rank. 

\begin{remark} The disturbance $\delta$ is related to the neutral part of system \eqref{eq:problem}. The disturbance component $\delta_0$ is related to Jensen inequality (see for instance \cite{gu2003stability}), a widely used inequality in the analysis of time-delay systems.
\end{remark}

We next need to propose a structure for the real-valued separator $\Theta$ such that inequality \eqref{eq:separateur} always holds for $\nabla \in \dnabla$. For $\Theta = \left[ \begin{smallmatrix} \Theta_{11} & \Theta_{12} \\ \Theta_{12}^{\top} & \Theta_{22} \end{smallmatrix} \right]$, the following structure is proposed:
\begin{equation} \label{eq:theta}
	\begin{array}{lll}
		\Theta_{11} = \text{diag}\left( 0_n, -Q, R(1 - \alpha^2) \gamma^2, -\tau^2 S \right), & \quad
		\Theta_{12} = \text{diag}\left( -P, 0, -R \gamma,  0 \right),  & \quad
		\Theta_{22} = \text{diag}\left( 0_n, Q, R, S \right),
	\end{array}
\end{equation}
with $P \in \mathbb{S}^n_+$ and $Q, R, S \in \mathbb{R}^+$. This selection is not new and is used in the examples presented in \cite{668829,peaucelle2007quadratic}. With this specific structure, for all $s \in \mathbb{C}^+$, the following holds:
\begin{equation}
	\!\left[ \begin{matrix} I \\ \nabla(s) \end{matrix} \right]^* \Theta \left[ \begin{matrix} I \\ \nabla(s) \end{matrix} \right] =
	 \begin{array}[t]{l}
		\text{diag}\left\{ - 2 P \mathfrak{R}\text{e} ( s^{-1} ), Q (\left\lvert e^{-\tau s} \right\rvert^2 -1) , \right. \quad
		\underbrace{R\left( (1 - \alpha^2) \gamma^2 - 2 \gamma \mathfrak{R}\text{e}\left( \delta(s) \right) + \left\lvert \delta(s) \right\rvert^2 \right)}_{\delta_{-1}(s)}, \quad
		\left. S \left( \left\lvert \delta_0(s) \right\rvert^2 -\tau^2\right) \right\}.
	\end{array}
\end{equation}

The third diagonal block can be written differently: $\delta_{-1}(s) = R\left( \lvert \delta(s) - \gamma \rvert^2 - \alpha^2 \gamma^2 \right)$. Equation \eqref{eq:def1QS} implies that $\delta$ is inside a circle, its center is $\gamma = \frac{1+\alpha}{1 - \alpha^2}$ and its radius is $\frac{(1 + \alpha) \lvert \alpha \rvert}{1 - \alpha^2} = \lvert \alpha \rvert \gamma$, guaranteeing that $\delta_{-1}(s) \infeq 0$.

Noticing also that $\forall s \in \mathbb{C}^+$, $\lvert e^{-\tau s}\rvert \infeq 1$, $\left\lvert \frac{1 - e^{-\tau s}}{s} \right\rvert \infeq \tau$, we get inequality \eqref{eq:separateur}.
All these considerations lead to the following stability theorem:
\begin{theo} \label{sec:theoQS} If there exist $P \in \mathbb{S}^n_+$ and $Q, R, S \in \mathbb{R}^+$ such that LMI \eqref{eq:LMI} holds for $\Theta$ defined in \eqref{eq:theta}, then system \eqref{eq:problem} is input-output stable.
\end{theo}

\subsection{Quadratic Separation - Extended stability analysis}

\subsubsection{Motivations and main theorem}

The studies of time-delay systems during the last few years focused on reducing the conservatism of stability theorems (see \cite{6760001} for example). Indeed, it has been showed in \cite{6760001,seuret:hal-01065142,briat2011convergence} that Jensen's inequality leads to conservative results. The idea developed here is to enrich $\Omega$ and $Z$ wisely to improve Theorem \ref{sec:theoQS}. Following this idea, we aim at capturing the infinite dimensional behavior of this system, described in equation \eqref{eq:ux}, by adding new signals into the framework.

Indeed, quadratic separation clearly shows that the addition of more information will lead to a smaller kernel of $\left[ \begin{smallmatrix} \mathcal{E} & -\mathcal{A} \end{smallmatrix} \right]$ followed by an improvement of the stability criterion at a price of a more complex $\mathcal{A}$, $\mathcal{E}$ and $\nabla$. Following the methodology described in \cite{seuret:hal-01065142}, the new signals are projections of the infinite dimensional state $u$. This state is projected onto the orthogonal basis of shifted Legendre polynomials $\left\{ \mathcal{L}_k \right\}_{k \in [0, N]}$. Some useful properties of these polynomial are reminded in the sequel. For more information, the reader can refer to \cite{courant1966courant}.

\begin{theo} \label{sec:theoQSN}
For a given $N \in \mathbb{N}$, if there exist $P_N \in \mathbb{S}^{n+N}_+$ and $Q, R, S \in \mathbb{R}^+$ such that the LMI:
\begin{equation} \label{eq:LMIN}
	{\left[ \begin{matrix} \mathcal{E}_N & - \mathcal{A}_N \end{matrix} \right]^{\perp}}^{\top} \left[ \begin{matrix} \Theta_{N,1} & \Theta_{N,2} \\ \Theta_{N,2}^{\top} & \Theta_{N,3} \end{matrix} \right] \left[ \begin{matrix} \mathcal{E}_N & - \mathcal{A}_N\end{matrix} \right]^{\perp} \succeq 0,
\end{equation}
holds for
\begin{equation} \label{eq:ThetaN}
	\begin{array}{lll}
		\Theta_{N,1} = \text{diag} \left( 0_{n+N}, -Q, R(1-\alpha^2) \gamma^2, - \tau^2 S \right), & \quad
		\Theta_{N,2} = \text{diag} \left( -P_N, 0, -R \gamma, 0_{1, N+1} \right), & \quad
		\Theta_{N,3} = \text{diag} \left( 0_{n+N}, Q, R, S I_{N+1} \right),
	\end{array}
\end{equation}
\begin{equation}
	\begin{array}{l}
		\!\! \mathcal{E}_N =\left[ \begin{matrix}
			I_{n} & 0_{n, N} & 0_{n, 1} & 0_{n, 1} & 0_{n, 1} \\
			0_{N, n} & I_{N} & 0_{N,1} & 0_{N,1} & 0_{N, 1} \\
			0_{1,n} & 0_{1,N} & 1 & 0 & 0 \\
			0_{1,n} & 0_{1, N} & 0 & 1 & 0 \\
			-K & 0_{1, N} & 0 & 0 & 1 \\
			0_{N+1,n} & 0_{N+1, N} & 1_{N+1, 1} &  -\mathbb{1}_{N+1} & 0_{N+1,1}
		\end{matrix} \right], \quad \mathcal{A}_N = \left[ \begin{matrix}
			A & 0_{n, N} & 0_{n, 1} & B & 0_{n, N+1} \\
			0_{N,n} & 0_{N} & 0_{N, 1} & 0_{N, 1} & \tilde{I}_N\!{\scriptstyle (1:N,:)} \\
			K & 0_{1,N} & 0 & 0 & 0_{1, N+1} \\
			0_{1,n} & 0_{1,N} & 1 & 0 & 0_{1,N+1} \\
			0_{1,n} & 0_{1,N} & 0 & 0 & 0_{1, N+1} \\
			0_{N+1,n} & L_N\!{\scriptstyle (:,1:N)} & 0_{N+1,1} & 0_{N+1,1} & \tilde{I}_N
		\end{matrix} \right], \\
	\end{array}
\end{equation}
with
\begin{equation} \label{eq:lik}
	\begin{array}{ll}
		\mathbb{1}_N = \left[ \begin{matrix} (-1)^0 & \cdots &(-1)^k & \cdots & (-1)^{N-1} \end{matrix} \right]^{\top}\!\!\!\!, \quad \quad \quad \quad \quad &
		\tilde{I}_N = \text{diag}\left( \left\{ \frac{1}{\sqrt{2k + 1}} \right\}_{k \in [0, N]} \right), \\
		L_N = \left[ \ell_{ij} \right]_{i,j \in [0, N]}, &
		\ell_{ik} = \left\{ \begin{array}{ll}
			0, & \text{ if } k \supeq i, \\
			(2k+1) \left(1-(-1)^{k+i} \right) c, \quad \quad \quad & \text{otherwise},
		\end{array} \right.
	\end{array}
\end{equation}
then system \eqref{eq:problem} is input/output stable.
\end{theo}

\begin{remark} The case $N=0$ leads to Theorem~\ref{sec:theoQS}. Theorem~\ref{sec:theoQSN} also introduces a hierarchy of stability conditions. In other words, if system \eqref{eq:problem} is proven to be exponentially stable using LMI \eqref{eq:LMIN} for a given $N = N_0$, then for all $N \supeq N_0$, LMI \eqref{eq:LMIN} also assesses the same stability\cite{seuret:hal-01065142}. \end{remark}

Compared to classical stability analysis approaches using the Lyapunov Stability (LS) obtained with the Lyapunov functional of \cite{besselString}, this method also uses an LMI solver. It usually results in LMIs with less decision variables than other techniques, which is critical when handling high dimensional systems. Indeed, the difference between LS and Quadratic Separation (QS) is the number of variables. This is due to the two dimensions state extension in LS. The double state extension leads to an increase of the number of variables and then to a slower computation. Table~\ref{tab:compVar} and Figure~\ref{fig:compVar} show that considering $N = 2$ with LS has more decision variables than considering QS at $N = 5$.

\begin{table}
	\centering
	\begin{tabular}{c|c|c}
		& Lyapunov Functional & Quadratic Separation \\
		\hline
		& $\cfrac{n^2 + n}{2} + 2 (N^2+n) + 5N + Nn + 9$ & $\cfrac{n^2 + n + N^2 + N}{2} + Nn + 3$ \\
		$N = 0$ & $27$ & $13$ \\
		$N = 2$ & $61$ & $24$ \\
		$N = 5$ & $142$ & $48$		
	\end{tabular}
	\caption{Number of variables for QS and LS for $n = 4$ and an order $N$.}
	\label{tab:compVar}
\end{table}
\begin{figure}
	\centering
	\includegraphics[width=9cm]{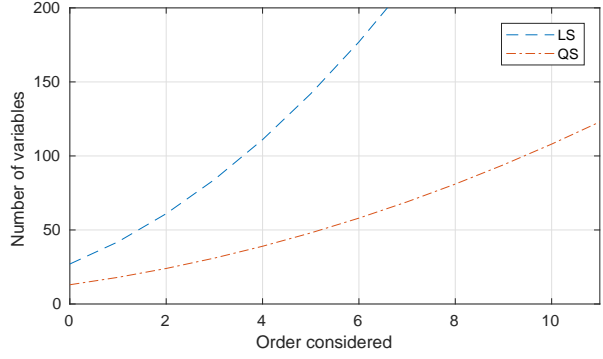}
	\caption{Number of variables for QS and LS for $n = 4$ and an order $N$.}
	\label{fig:compVar}
\end{figure}

\subsubsection{Proof of Theorem \ref{sec:theoQSN}}

The proof of the theorem is divided into 2 parts. First, we will introduce some tools used in the sequel and then the proof itself is derived.

\paragraph{Legendre Polynomials and frequency-bessel inequality}
Before going further, the tools used in the sequel are introduced. For the projections, the scalar product is the canonical inner product of the Hilbert space $L^2([-\tau, 0], \mathbb{C})$ and is denoted $\langle f, g \rangle = \int_{-\tau}^0 f^*(\theta) g(\theta) d\theta$ for $f, g \in L^2([-\tau, 0], \mathbb{C})$. $\| \cdot \|_2$ is the norm given by the previous inner-product. The shifted Legendre polynomials are defined as follows:
\begin{equation}
	\mathcal{L}_k(u) = (-1)^k \sum_{l = 0}^k (-1)^l \left( \begin{matrix} k \\ l \end{matrix} \right) \left( \begin{matrix} k + l \\ l \end{matrix} \right) \left( \frac{u+\tau}{\tau} \right)^l.
\end{equation}

Legendre polynomials have been chosen because of their interesting properties (see \cite{seuret:hal-01065142}).  First, because the evaluation of each Legendre polynomial at its boundaries $0$ and $-\tau$ is simple. Secondly, as it is a polynomial basis, the following differentiation rule applies for $N \in \mathbb{N}, x \in (-\tau, 0)$:
\begin{equation} \label{eq:derivationRule}
	\left[ \begin{matrix} \dot{\mathcal{L}}_0(x) & \cdots & \dot{\mathcal{L}}_N(x) \end{matrix} \right]^{\top} = L_N \left[ \begin{matrix} \mathcal{L}_0(x) & \cdots & \mathcal{L}_N(x) \end{matrix} \right]^{\top},
\end{equation}
where $L_N$ is given in \eqref{eq:lik}.
Finally, the Legendre polynomials family is orthogonal with respect to the inner product $\langle \cdot, \cdot \rangle$ and it is possible to use Bessel inequality. This inequality is the key part to build a separator and a version adapted to this problem is proposed below:
\begin{lemma} \label{sec:lemmaBessel}
	Let $\tau > 0$ and $N \in \mathbb{N}$, then the following inequality holds:
	\begin{equation} \label{eq:bessel}
		\forall s \in \mathbb{C}^+,  \quad  \delta_N^*(s) \delta_N(s)  \infeq \tau^2 ,
	\end{equation}
	with $\delta_N(s) = \sqrt{\tau}\left[ \begin{matrix} \left\langle e^{\theta s}, \frac{\mathcal{L}_0(\theta)}{\| \mathcal{L}_0 \|_2} \right\rangle & \cdots & \left\langle e^{\theta s}, \frac{\mathcal{L}_{N}(\theta)}{\| \mathcal{L}_{N} \|_2} \right\rangle \end{matrix} \right]^{\top}\!\!$.
\end{lemma}
\begin{proof} Let $s \in \mathbb{C}^+$, Bessel inequality applied to function $\theta \mapsto \exp(\theta s)$ gives: $
	\sum_{k=0}^{N} \frac{1}{\| \mathcal{L}_k \|_2^2} \left\lvert \left\langle e^{\theta s}, \mathcal{L}_k(\theta) \right\rangle \right\rvert^2 \infeq \| e^{\theta s} \|_2^2$.\\
An identification with $\delta_N$ leads to: $\delta_N^*(s) \delta_N(s) \infeq \tau \left\lvert \int_{-\tau}^0  e^{\theta (s + s^*)} d\theta \right\rvert$.
As $s \in \mathbb{C}^+$, the right hand side is bounded by $\tau^2$ and that ends the proof.
\end{proof}

\paragraph{Quadratic Separation formulation of system \eqref{eq:problem}}
In order to ease the comparison with time-delay systems, we introduce a new variable $\theta$ defined as follows: $\theta = -\tau x = - \frac{x}{c}$.
If the neutral part is not taken into consideration, a new infinite dimensional state can be defined: $\tilde{U}(\theta, s) = U(- c \theta, s) \delta^{-1}(s)$. Using equation \eqref{eq:ux}, its expression is then:
\begin{equation} \label{eq:stateU}
	\tilde{U}(\theta,s) = \frac{e^{\theta s} + \alpha e^{-(\theta+2\tau)s} }{1 + \alpha} KX(s),
\end{equation}
with $t > 0$ and $\theta \in (-\tau, 0)$. There are several starting points to enrich $\Omega$, and the one we choose in this paper is to consider only the projection of the first part of the infinite dimensional state, i.e. $\theta \mapsto e^{s \theta} K X(s)$. We justify our choice in Remark~\ref{sec:otherStateExtension}.

Let $N \supeq 0$ be the number of projections used.
The purpose is now to build $\Omega_N$ and $Z_N$ such that:
\begin{equation} \label{eq:sepQuadN}
	\begin{array}{l}
		\Omega_N(s) = \nabla_N(s) Z_N(s), \\
		\mathcal{E}_N Z_N(s) = \mathcal{A}_N \Omega_N(s), \\
	\end{array}
\end{equation}
with $\mathcal{E}_N$ full column-rank where $\Omega_N$ and $Z_N$ are extended versions of $\Omega$ and $Z$. The projections for $k \in [0, N]$ are introduced as follow:
\begin{equation} \label{eq:projections}
	\begin{array}{rlrl}
		\chi_k(s) &= \left\langle e^{\theta s} , \mathcal{L}_k(\theta) \right\rangle K X(s), \quad \quad & \Chi_N(s) &= \left[ \begin{matrix} \chi_0(s) & \cdots & \chi_{N-1}(s) \end{matrix} \right]^{\top}, \\
		\nu_k(s) &= s \sqrt{2k +1} \chi_k(s), & \Nu_N(s) &= \left[ \begin{matrix} \nu_0(s) & \cdots & \nu_{N}(s) \end{matrix} \right]^{\top}.
	\end{array}
\end{equation}

The new state in Laplace domain is composed of $X(s)$ and the projections of $\theta \mapsto KX(s)e^{\theta s}$ on the orthogonal basis of Legendre polynomials using the projections defined in equation \eqref{eq:projections}. Noting that $\mathcal{L}_k(-\tau) = (-1)^k$ and $\mathcal{L}_k(0) = 1$, the derivation rule in equation \eqref{eq:derivationRule} and an integration by parts give the following result:
\begin{equation}
	s \chi_k(s) =\left( 1 - (-1)^k e^{-\tau s} - \int_{-\tau}^0 e^{\theta s} \mathcal{L}'_k(\theta) d\theta \right) K X(s) = \left(1 - (-1)^{k} e^{-\tau s}\right) KX(s) - \sum_{i=0}^{k-1} \ell_{ik} \chi_i(s),
\end{equation}
for $k \in [0, N-1]$ and $\ell_{ik}$ as defined in equation \eqref{eq:lik}.
Noting that $|| \mathcal{L}_k ||^2 = \tau (2k +1)^{-1}$, $\Nu_N$ can be expressed as follows: $\Nu_N(s) =  \delta_N(s) sKX(s)$.
The new state vector for the quadratic separation are:
\begin{equation}
	\begin{array}{ll}
		z_N(t) = \left[ \begin{matrix} \dot{X}^{\top}\!(t)\! & \dot{\Chi}^{\top}_N(t)\! & K X(t)\! & K X (t-\tau)\! & K \dot{X}\!(t) \end{matrix} \right]^{\top}\!\!\!, & \quad \quad
		\omega_N(t) = \left[ \begin{matrix} X^{\top}\!(t)\! & \Chi_N^{\top}(t)\! & K X(t-\tau)\! & u(1,t)\! & \Nu^{\top}_{N}(t) \end{matrix} \right]^{\top}\!\!\!.
	\end{array}
\end{equation}

With the previous signals, equalities \eqref{eq:sepQuadN} hold for $\mathcal{A}_N$, $\mathcal{E}_N$ defined in the theorem and $\nabla_N(s) = \text{diag}\left(s^{-1} I_{n+N}, e^{-\tau s}, \delta(s), \delta_N(s) \right)$.
Now, we need to find a separator $\Theta_N$. The state extension is based on projections on an orthogonal basis such that the Bessel inequality holds. The result presented in Lemma \ref{sec:lemmaBessel} guarantees that $\Theta_N$ proposed in the theorem is a solution to LMI \eqref{eq:separateur}.

\begin{remark}\label{sec:remQS} One of the main problem of this method comes from the inclusion of $\delta$ into a disk. This may not be a convenient bound for systems with a high reflexion coefficient $\alpha$. So for $c$ and $c_0$ smaller or $c$ and $c_0$ larger, weaker results are expected. \end{remark}

%

\begin{remark}[(Another state extension)] \label{sec:otherStateExtension} We justify the choice of projecting only part of the state $\tilde{U}(\cdot,t)$. First, in equation \eqref{eq:stateU}, the state $\tilde{U}$ at a given $\theta$ and $t$ is made up of two contributions. The first one is a result of a wave going forward $\tilde{U}_1(\theta, s) = e^{s \theta} KX(s)$ and another one going backward $\tilde{U}_2(\theta, s) = e^{-s (\theta+2\tau)} KX(s)$. We decided here to project only $\tilde{U}_1$. Another option is to consider the two components $\tilde{U}_1$ and $\tilde{U}_2$ independently and, then, to enrich the state by two projections at each order. This leads to an increased number of variables and, unfortunately, similar performances.

Indeed, to keep the same number of variables, one solution would be to project the whole state $U(x,s)$ on the same basis of Legendre polynomials. Considering $\theta = -\tau x$, we get the following equations for $\theta \in (-\tau, 0)$:
\begin{equation}
	\tilde{U}(\theta, s) = U(\theta, s) \delta^{-1}(s) = \left( e^{s \theta} + \alpha e^{-s(\theta + 2 \tau)} \right) U(0,s).
\end{equation}

Using the dot product defined earlier, we get:
\begin{equation}
	\frac{\chi_k(s)}{U(0,s)} = \left\langle \frac{\tilde{U}(\theta, s)}{U(0,s)}, \mathcal{L}_k(\theta) \right\rangle = \left\langle e^{s\theta}, \mathcal{L}_k(\theta) \right\rangle
	+ \alpha e^{-s\tau} \int_{-\tau}^0 e^{-s (\theta + \tau)} \mathcal{L}_k(\theta) d\theta
	= \left\langle e^{s \theta}, \mathcal{L}_k(\theta) \right \rangle
	+ \alpha e^{-s \tau} \left\langle e^{s \theta}, \mathcal{L}_k(- \theta - \tau) \right\rangle.
\end{equation}

Noting that the shifted Legendre polynomials are symmetric/antisymmetric relative to $-\frac{\tau}{2}$ which means we have $\mathcal{L}_k(- \theta - \tau) = (-1)^k \mathcal{L}_k(\theta)$ and hence we get:
\begin{equation}
	\chi_k(s) = \left\langle e^{s \theta}, \mathcal{L}_k(\theta) \right \rangle \left( 1 + (-1)^k\alpha e^{-s \tau} \right) U(0,s).
\end{equation}

The problem comes from the Bessel-like inequality in \eqref{eq:bessel} and now reads:
\begin{equation}
	\sum_{k=0}^N (2k+1) \left|\frac{\chi_k(s)}{U(0,s)}\right|^2 \infeq \tau^2 (1 + |\alpha|)^2.
\end{equation}

This inequality is not an optimal bound because it will not become an equality for $N \to \infty$, which was the case previously.\end{remark}

\section{Examples}

In this part, we aim at presenting the difference between the exact stability area and the one obtained using robust stability theorems. The same system has also been studied in \cite{besselString} using a Lyapunov-based stability approach. It also uses a hierarchy of LMI conditions, so the notation ``LS, $N = i$'' refers to the stability obtained using the other methodology for an order $i$. A comparison of efficiency between all the methods is also presented. 

The estimation of the stability area of system \eqref{eq:problem} is provided on some chosen examples with different behaviors. We consider the interconnection of a stable wave equation ($c_0 > 0$) interconnected with different ODE systems. First, the interconnection with a stable finite dimension LTI system is proposed.
Then, an unstable ODE is interconnected. We then aim at proving that there exist unstable systems stabilized thanks to the wave equation.
To finish, we analyze a system known to possess stability pockets for $c c_0 = 1$. The LMI solver used in the examples is ``sdpt3'' with Yalmip\cite{1393890}.

\subsection{$A$ and $A+BK$ Hurwitz with $\|H\|_{\infty} < 1$}

This first example is borrowed from \cite{besselString} where the $K$ matrix has been slightly modified to get $\|H\|_{\infty} < 1$ such that the Small-Gain Theorem can be applied. System \eqref{eq:problem} is proposed with the following matrices:
\begin{equation} \label{eq:simu3}
	\begin{array}{ccc}
		A = \left[ \begin{matrix} -2 & 1 \\ 0 & -1 \end{matrix} \right], &
		B = \left[ \begin{matrix} 1 \\ 1 \end{matrix} \right], &
		K = \left[ \begin{matrix} 0 & -\frac{20}{21} \end{matrix} \right].
	\end{array}
\end{equation}

It is easy to verify that $A$ and $A+BK$ are indeed Hurwitz. Moreover, the infinity norm of the open-loop system is less than 1, then Theorem \ref{smallGain} applies and $\lim_{c_0 \to \infty} c_{min}(c_0) = 0$. Here we are after computing the $c_{min}$, which is the smallest $c$ such that for all $c > c_{min}$, the system is input/output stable. While for the Small-Gain theorem, we have proven this bound, the remaining approaches will reveal the bounds through computations.

The results with quadratic separation (QS) and the method with Lyapunov-based stability (LS) at order $N = 0$ developed in \cite{besselString} are superposed in the chart of Figure \ref{fig:simu3}. As expected, the small gain theorem provides the worst estimation of the stability area but with stronger properties. The QS and LS approaches provide similar results. The QS does not use an extended state so the results for high $c_0$ are further from the CTCR curve than the one obtained with LS.\\ Finally, the result obtained with CTCR shows a non-continuous behavior for small $c c_0$ and this observation discourages us to use the $(c_0, c)$ chart to estimate the stability area of system \eqref{eq:problem}. This is not a numerical issue, but rather indicates some singularities regarding system poles crossing the imaginary axis for certain values of $c_0$. This can be remedied by viewing the stability on a different domain of parameters. Indeed, an appropriate choice of system coordinates is $(c c_0, c^{-1})$. In order to make a comparison with time-delay systems, the delay $c^{-1}$ needs to be considered as one of the axes. It is natural to add $c c_0$ as the other variable of interest as it pilots the behavior of the system. Then, on the line $c c_0 = 1$, we get a time-delay system (see the state-space representation \eqref{eq:ss} and Remark~3), and it is possible to compare the results with the literature. From this point, the mapping $(c c_0, c^{-1})$ is used.


\begin{figure}
	\centering
	\includegraphics[width=13cm]{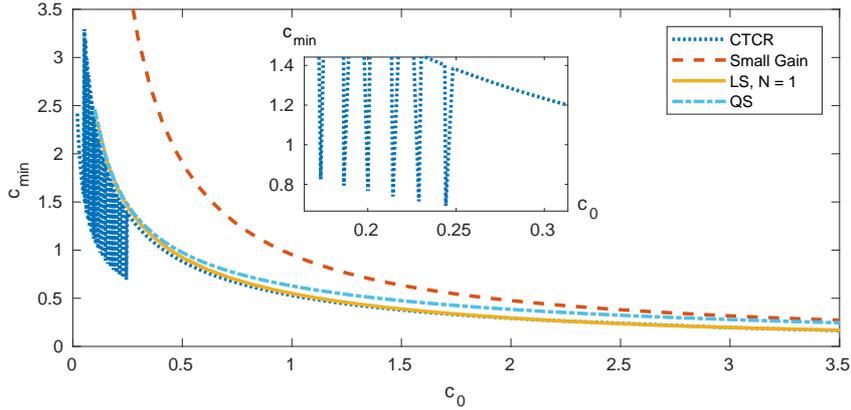}
	\vspace{-0.2cm}
	\caption{Minimum wave speed $c_{min}$ as a function of $c_0$ for system \eqref{eq:problem} to be stable. The values for $A$, $B$ and $K$ are given by equation \eqref{eq:simu3}.}
	\label{fig:simu3}
\end{figure}

\subsection{$A$ and $A+BK$ are not Hurwitz}

The following system is borrowed from \cite{6760001} where 
\begin{equation} \label{eq:simu4}
	\begin{array}{ccc}
		A = \left[ \begin{matrix} 0 & 1 \\ -2 & 0.1 \end{matrix} \right], &
		B = \left[ \begin{matrix} 0 \\ 1 \end{matrix} \right], &
		K = \left[ \begin{matrix} 1 & 0 \end{matrix} \right].
	\end{array}
\end{equation}

It has been shown in the cited study that for $c_1= c c_0 = 1$, it is stable for a sufficiently large delay. This is a very different theoretical case compared to the previous one. For $\tau = 0$, uncontrolled system has two unstable poles, which means that a simple output feedback cannot stabilize the system. This is evident from both $A$ and $A+BK$ being unstable. Notice first that this system for $c c_0 = 1$ has been studied in \cite{6760001}, and CTCR recovers the same exact stability area.

QS at $N = 0$ does not provide meaningful results, but, as we can see in Figure \ref{fig:compBesselCTCR2}, higher orders of QS detect a stability area which approaches the exact results by CTCR as the order increases; see the stability regions around $c_1 = 1$. Despite the state augmentation, it is not possible to recover the stability area and computations show that for $c c_0 < 0.5$, QS at any order does not assess stability. This can be a consequence of considering  an inappropriate bound on $\delta$.


\begin{figure}
	\centering
	\includegraphics[width=10cm]{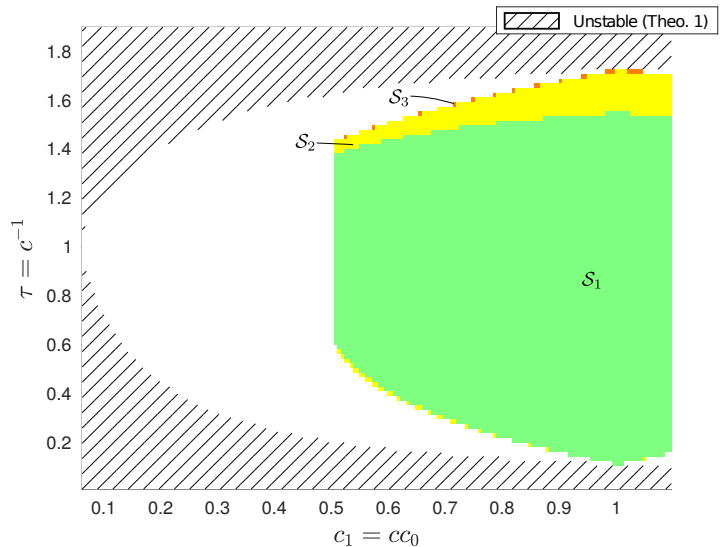}
	\caption{Stability areas for system \eqref{eq:problem} with matrices defined in \eqref{eq:simu4} and $c_1 = c c_0$ and $\tau = c^{-1}$ obtained using CTCR and Theorem \ref{sec:theoQSN}. The hatched area is the exact unstable area, the remaining area is the exact stable area as detected by CTCR. The colored areas are stated as stable for Theorem \ref{sec:theoQSN} at an order $N = 3$. The color scale represents the stability area depending on the order $N$ considered. The areas marked as $\mathcal{S}_i$ is the stable area up to an order $i$.}
	\label{fig:compBesselCTCR2}
\end{figure}

\subsection{An example with stability pockets}

For this last example, the system is taken from \cite{gu2003stability}. It is known to possess multiple stable intervals (pockets) along the delay axis for $c_1 = c c_0 = 1$. We investigate whether or not QS and the methodology of \cite{besselString} can detect these pockets. The system matrices are given by:
\begin{equation} \label{eq:simu5}
	A = \left[ \begin{matrix} 0 & 0 & 1 & 0 \\ 0 & 0 & 0 & 1 \\ -11 & 10 & 0 & 0 \\ 5 & -15 & 0 & -0.25 \end{matrix} \right], \ B = \left[ \begin{matrix} 0 \\ 0 \\ 1 \\ 0 \end{matrix} \right],  \ K = \left[ \begin{matrix} 1 \\ 0 \\ 0 \\ 0 \end{matrix} \right]^{\top} \!\!\!\!.
\end{equation}

We can compare the efficiency of LS in Figure~\ref{fig:compBesselCTCR}. The hierarchy property can be seen and the stability pockets are indeed recovered as the order increases. Figure \ref{fig:compQSCTCR} shows the stability result with QS at an order $N$. QS and LS at the same order provide similar results but with lower decision variables for QS, see Table~\ref{tab:compVar} and Figure~\ref{fig:compVar}.

Moreover, for small $c_1$, the reflection coefficient is not close to $0$ and then, as noted in Remark \ref{sec:remQS}, QS should have a worse estimation. That is why we get poor result on the left side compared to LS. But if $c_1$ is close to $1$, results similar to LS are obtained using QS with a subsequently lower number of decision variables.
That means at the same number of decision variables, QS can detect more stability pockets around $c_1 = 1$ but LS detects a wider stability area. When $4 \leq N \leq 7$, the stability area detected by QS is included in the one ensured by LS  at the price of more decision variables. Not matter the order, QS never detects the whole stability area.

\begin{figure*}
	\centering
	\subfloat[Theorem 2 from \cite{besselString}: Lyapunov functional]{\includegraphics[width=12cm]{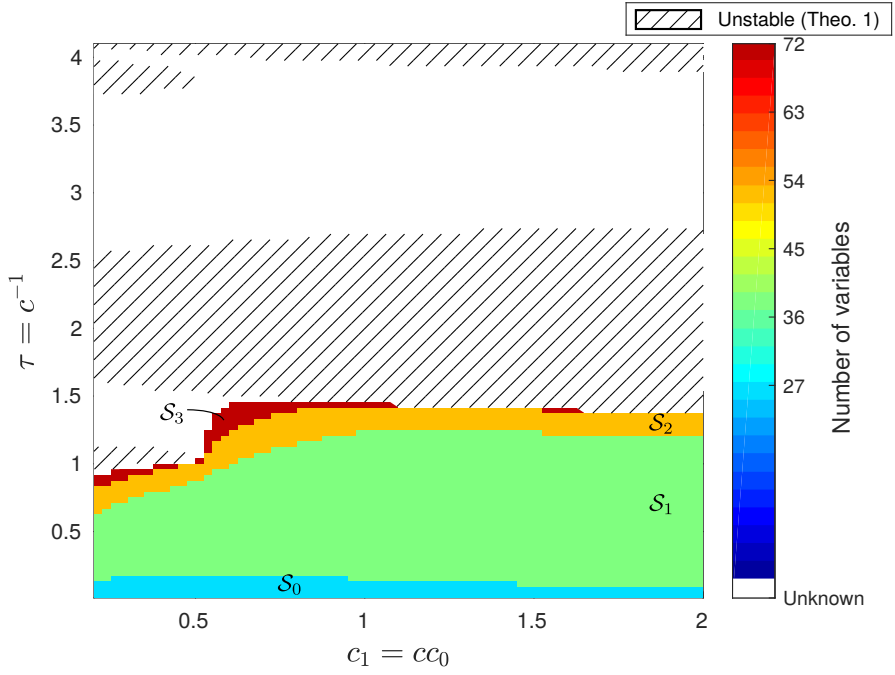} \label{fig:compBesselCTCR}} \\
	\subfloat[Theorem 4: Quadratic Separation]{\includegraphics[width=12cm]{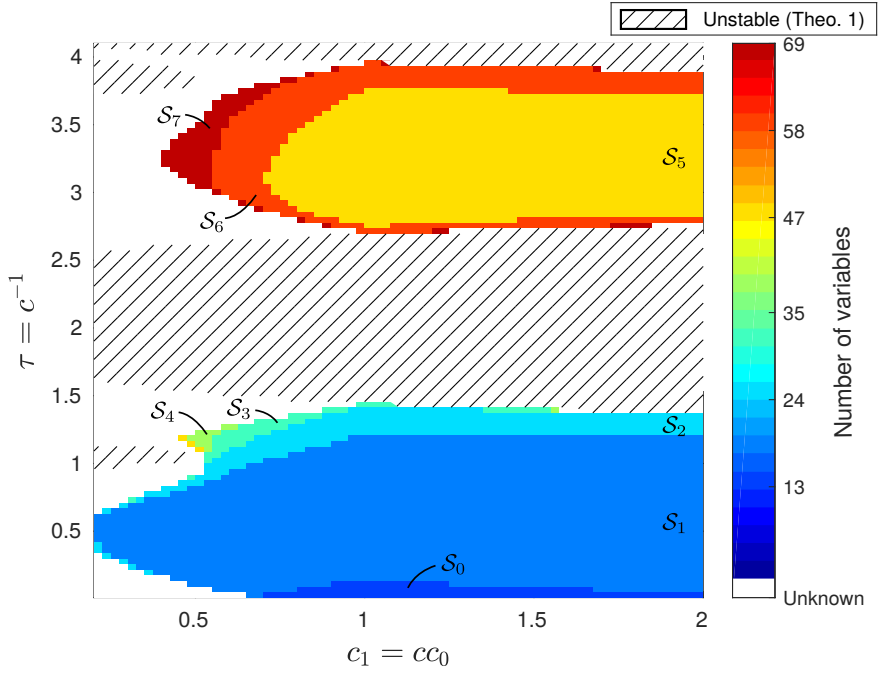} \label{fig:compQSCTCR}}
	\caption{Stability areas for system \eqref{eq:problem} with matrices defined in \eqref{eq:simu5} and $c_1 = c c_0$ and $\tau = c^{-1}$ obtained using CTCR, LS and QS. The hatched area is the exact unstable area, the white area corresponds to the unknown area according to the corresponding theorems but known to be stable with CTCR. The color scale represents the stability area depending on the number of variables considered. To enable comparison, $N$ varies between $0$ and $3$ for LS and between $0$ and $7$ for QS. The areas marked as $\mathcal{S}_i$ is the stable area up to an order $i$.}
\end{figure*}

\section{Conclusion}

In this paper, we studied using an exact method the input/output stability in a parameter region of interest for a system interconnecting an ODE with a string-type PDE equation. We also proposed two robust stability results: a simple one obtained with the Small-Gain theorem and another one using Quadratic Separation. The latter is based on a hierarchy of LMI conditions and is more conservative than some other approaches; but it proposes a subsequently reduced number of decision variables. A perspective would be to enhance the stability results for Quadratic Separation while keeping its low computational burden. Future work will be aimed at determining which approximation introduces the most conservatism in order to reduce the gap between the Lyapunov Stability analysis and the one made with Quadratic Separation.

\bibliographystyle{plain}
\bibliography{report_draft}%

\begin{thebibliography}{10}

\bibitem{Arcak2016}
M~Arcak, C~Meissen, and A~Packard.
\newblock {\em Stability of Interconnected Systems}.
\newblock Springer, 2016.

\bibitem{6760001}
Y~Ariba, F~Gouaisbaut, A~Seuret, and D~Peaucelle.
\newblock {Stability analysis of time-delay systems via Bessel inequality: A
  quadratic separation approach}.
\newblock {\em {International Journal of Robust and Nonlinear Control}},
  28(5):pp~1507--1527, March 2018.

\bibitem{besselString}
M~Barreau, A~Seuret, F~Gouaisbaut, and L~Baudouin.
\newblock Lyapunov stability analysis of a string equation coupled with an
  ordinary differential system.
\newblock {\em IEEE Transactions on Automatic Control}, 1, 2018,
  DOI:10.1109/TAC.2018.2802495.

\bibitem{bastin2016stability}
G.~Bastin and J.-M. Coron.
\newblock {\em Stability and boundary stabilization of 1-d hyperbolic systems},
  volume~88.
\newblock Springer, 2016.

\bibitem{bellman1963differential}
R~E Bellman and K~L Cooke.
\newblock {\em Differential-difference equations}.
\newblock Academic Press, 1963.

\bibitem{beretta2002geometric}
E~Beretta and Y~Kuang.
\newblock Geometric stability switch criteria in delay differential systems
  with delay dependent parameters.
\newblock {\em SIAM Journal on Mathematical Analysis}, 33(5):1144--1165, 2002.

\bibitem{bresch2014output}
D~Bresch-Pietri and M~Krstic.
\newblock Output-feedback adaptive control of a wave {PDE} with boundary
  anti-damping.
\newblock {\em Automatica}, 50(5):1407--1415, 2014.

\bibitem{briat2011convergence}
C~Briat.
\newblock Convergence and equivalence results for the jensen's inequality -
  application to time-delay and sampled-data systems.
\newblock {\em IEEE Transactions on Automatic Control}, 56(7):1660--1665, 2011.

\bibitem{califano2017stability}
F~Califano, A~Macchelli, and C~Melchiorri.
\newblock Stability analysis of repetitive control: The port-hamiltonian
  approach.
\newblock (1):1894--1899, 2017.

\bibitem{courant1966courant}
R~Courant and D~Hilbert.
\newblock {\em Methods of mathematical physics}, volume~1 of {\em Wiley
  Classics Library}.
\newblock John Wiley \& Sons, Inc, New York, 1989.

\bibitem{curtain1995introduction}
R~F Curtain and H~J Zwart.
\newblock {\em An introduction to infinite-dimensional linear systems theory},
  volume~21 of {\em Texts in Applied Mathematics}.
\newblock Springer-Verlag, New York, 1995.

\bibitem{doetsch2012introduction}
G~Doetsch.
\newblock {\em Introduction to the theory and application of the Laplace
  transformation}.
\newblock Springer, 2012.

\bibitem{espitia2016event}
N~Espitia, A~Girard, N~Marchand, and C~Prieur.
\newblock Event-based control of linear hyperbolic systems of conservation
  laws.
\newblock {\em Automatica}, 70:275--287, 2016.

\bibitem{evans2010partial}
L~C Evans.
\newblock {\em Partial Differential Equations}.
\newblock Graduate studies in mathematics~. American Mathematical Society,
  2010.

\bibitem{flores2014dynamics}
G~Flores.
\newblock Dynamics of a damped wave equation arising from mems.
\newblock {\em SIAM Journal on Applied Mathematics}, 74(4):1025--1035, 2014.

\bibitem{Grabowski2001}
P~Grabowski and F~M Callier.
\newblock Boundary control systems in factor form: Transfer functions and
  input-output maps.
\newblock {\em Integral Equations and Operator Theory}, 41(1):1--37, Mar 2001.

\bibitem{gu2003stability}
K~Gu, J~Chen, and V~L Kharitonov.
\newblock {\em Stability of time-delay systems}.
\newblock Springer, 2003.

\bibitem{4099496}
B~Z Guo and C~Z Xu.
\newblock The stabilization of a one-dimensional wave equation by boundary
  feedback with noncollocated observation.
\newblock {\em IEEE Transactions on Automatic Control}, 52(2):371--377, Feb
  2007.

\bibitem{opac-b1084291}
J~K Hale and S~M~V Lunel.
\newblock {\em Introduction to functional differential equations}.
\newblock Applied mathematical sciences~. Springer-Verlag, New York, Berlin,
  Heidelberg, 1977.

\bibitem{hale1977theory}
{Hale, J~ K}.
\newblock {\em Theory of functional differential equations}.
\newblock Number vol~~3 in Applied Mathematical Sciences Series~. Springer
  Verlag, 1977.

\bibitem{HE2016146}
W~He and S~S Ge.
\newblock Cooperative control of a nonuniform gantry crane with constrained
  tension.
\newblock {\em Automatica}, 66:146 -- 154, 2016.

\bibitem{7476820}
W~He and S~Zhang.
\newblock Control design for nonlinear flexible wings of a robotic aircraft.
\newblock {\em IEEE Transactions on Control Systems Technology},
  25(1):351--357, Jan 2017.

\bibitem{6651788}
W~He, S~Zhang, and S~S Ge.
\newblock Adaptive control of a flexible crane system with the boundary output
  constraint.
\newblock {\em IEEE Transactions on Industrial Electronics}, 61(8):4126--4133,
  Aug 2014.

\bibitem{Wu20142787}
W~Huai-Ning and W~Jun-Wei.
\newblock Static output feedback control via {PDE} boundary and {ODE}
  measurements in linear cascaded {ODE}-beam systems.
\newblock {\em Automatica}, 50(11):2787 -- 2798, 2014.

\bibitem{668829}
T~Iwasaki and S~Hara.
\newblock Well-posedness of feedback systems: insights into exact robustness
  analysis and approximate computations.
\newblock {\em IEEE Transactions on Automatic Control}, 43(5):619--630, May
  1998.

\bibitem{jin2018stability}
C~Jin, K~Gu, S.~I Niculescu, and I~Boussaada.
\newblock Stability analysis of systems with delay-dependent coefficients.
\newblock {\em Phd Thesis, IEEE Access}, 2018.

\bibitem{kolmanovskii2013introduction}
V~Kolmanovskii and A~Myshkis.
\newblock {\em Introduction to the theory and applications of functional
  differential equations}, volume 463.
\newblock Springer, 2013.

\bibitem{krstic2011}
M~Krstic.
\newblock Dead-time compensation for wave/string pdes.
\newblock 133:4458--4463, 12 2009.

\bibitem{krstic2009delay}
M~Krstic.
\newblock {\em Delay compensation for nonlinear, adaptive, and {PDE} systems}.
\newblock Springer, 2009.

\bibitem{Lagnese1983163}
J~Lagnese.
\newblock Decay of solutions of wave equations in a bounded region with
  boundary dissipation.
\newblock {\em Journal of Differential Equations}, 50(2):163 -- 182, 1983.

\bibitem{1393890}
J~L{\"o}fberg.
\newblock Yalmip : a toolbox for modeling and optimization in matlab.
\newblock {\em Computer Aided Control Systems Design, 2004 IEEE International
  Symposium}, pages 284--289, Sept 2004.

\bibitem{louw2012forced}
T~Louw, S~Whitney, A~Subramanian, and H~Viljoen.
\newblock Forced wave motion with internal and boundary damping.
\newblock {\em Journal of applied physics}, 111:14702--147028, 01 2012.

\bibitem{luo2012stability}
Z~H Luo, B~Z Guo, and {\"O}~Morg{\"u}l.
\newblock {\em Stability and stabilization of infinite dimensional systems with
  applications}.
\newblock Springer, 2012.

\bibitem{RNC:RNC1611}
T~Meurer and A~Kugi.
\newblock Tracking control design for a wave equation with dynamic boundary
  conditions modeling a piezoelectric stack actuator.
\newblock {\em International Journal of Robust and Nonlinear Control},
  21(5):542--562, 2011.

\bibitem{morgul1994dynamic}
{\"O}~Morg{\"u}l.
\newblock A dynamic control law for the wave equation.
\newblock {\em Automatica}, 30(11):1785--1792, 1994.

\bibitem{morgul1998stabilization}
{\"O}~Morg{\"u}l.
\newblock Stabilization and disturbance rejection for the wave equation.
\newblock {\em IEEE Transactions on Automatic Control}, 43(1):89--95, 1998.

\bibitem{Morgül2002731}
{\"O}~Morg{\"u}l.
\newblock An exponential stability result for the wave equation.
\newblock {\em Automatica}, 38(4):731 -- 735, 2002.

\bibitem{niculescu2001delay}
S~I Niculescu.
\newblock {\em Delay effects on stability: a robust control approach}, volume
  269.
\newblock Springer, 2001.

\bibitem{olgac2002exact}
N~Olgac and R~Sipahi.
\newblock An exact method for the stability analysis of time-delayed linear
  time-invariant ({LTI}) systems.
\newblock {\em IEEE Transactions on Automatic Control}, 47(5):793--797, 2002.

\bibitem{olgac2004practical}
N~Olgac and R~Sipahi.
\newblock A practical method for analyzing the stability of neutral type
  {LTI}-time delayed systems.
\newblock {\em Automatica}, 40(5):847--853, 2004.

\bibitem{peaucelle2007quadratic}
D~Peaucelle, D~Arzelier, D~Henrion, and F~Gouaisbaut.
\newblock Quadratic separation for feedback connection of an uncertain matrix
  and an implicit linear transformation.
\newblock {\em Automatica}, 43(5):795--804, 2007.

\bibitem{rekasius}
Z.~V Rekasius.
\newblock A stability test for systems with delays.
\newblock {\em Proceedings of the joint automatic control conference}, TP9-A,
  1980.

\bibitem{seuret:hal-01065142}
A~Seuret and F~Gouaisbaut.
\newblock Hierarchy of {LMI} conditions for the stability analysis of time
  delay systems.
\newblock {\em Systems and Control Letters}, 81:1--7, July 2015.

\bibitem{sipahi2011stability}
R~Sipahi, S~I Niculescu, C~T Abdallah, W~Michiels, and K~Gu.
\newblock Stability and stabilization of systems with time delay.
\newblock {\em IEEE Control Systems}, 31(1):38--65, 2011.

\bibitem{sipahi2003degenerate}
R~Sipahi and N~Olgac.
\newblock Degenerate cases in using the direct method.
\newblock {\em ASME 2003 International Design Engineering Technical Conferences
  and Computers and Information in Engineering Conference}, pages 2201--2210,
  2003.

\bibitem{tang2011state}
S~Tang and C~Xie.
\newblock State and output feedback boundary control for a coupled {PDE}--{ODE}
  system.
\newblock {\em Systems \& Control Letters}, 60(8):540--545, 2011.

\bibitem{tucsnak2009observation}
M~Tucsnak and G~Weiss.
\newblock {\em Observation and control for operator semigroups}.
\newblock Springer, 2009.

\bibitem{van2014port}
A~van~der Schaft and D~Jeltsema.
\newblock Port-hamiltonian systems theory: An introductory overview.
\newblock {\em Foundations and Trends in Systems and Control}, 1(2-3):173--378,
  2014.

\bibitem{7984225}
J~Wo\'zniak and B~Niesterowicz.
\newblock Input-output stability analysis of a slowly rotating rotor with
  friction on one end.
\newblock {\em 25th Mediterranean Conference on Control and Automation (MED)},
  (1):847--851, July 2017.

\end{thebibliography}
%
%

\end{document}